\documentclass[11pt]{article}
 \usepackage{amsmath,amsfonts,amssymb,amsthm,authblk} %
\usepackage [normalem]{ulem} %
\usepackage[utf8]{inputenc}

\newtheorem{theorem}{Theorem}[section] \newtheorem{definition}[theorem]{Definition}
 \newtheorem{remark}[theorem]{Remark} \newtheorem{lemma}[theorem]{Lemma}

  { \end{itemize} }
  { \end{enumerate} }

\newcommand{\qqquad}[0]{\qquad\qquad}

\newcommand{\NT}[0]{\notag}

\newcommand{\HW}[1]{} 

\newcommand{\IDEE}[1]{}
 
\allowdisplaybreaks[3]
\newcommand{\EMMM}[1]{\underline{#1}} 
\renewcommand{\em}[1]{{\EMMM{#1}}}
\def\ED{\end{document}}

\newenvironment{TC} {\left \{\begin{array}{ll}} {\end{array} \right.}
\newcounter{pcounter} 
 
 \newcommand{\BS}[0]{\backslash}

\renewcommand{\t}{\tilde}
\newcommand{\p}{\partial}

\newcommand{\DD}{{\mathcal D}}
\newcommand{\EE}{{\mathcal E}}

\newcommand{\HH}{{\mathcal H}}

\newcommand{\LL}{{\mathcal L}}
\newcommand{\MM}{{\mathcal M}}

\newcommand{\QQ}{{\mathcal Q}}

\newcommand{\YY}{{\mathcal Y}}
\newcommand{\ZZ}{{\mathcal Z}}

\newcommand{\R}{\mathbb{R}}
\newcommand{\N}{\mathbb{N}}

\newcommand{\eps}{\epsilon}

\renewcommand{\phi}{\varphi}

\newcommand{\Ome}{\Omega}

\newcommand{\nin}{\not\in}

\DeclareMathOperator*{\cof}{cof}

\DeclareMathOperator*{\supp}{supp}

\DeclareMathOperator*{\SO}{SO}

\DeclareMathOperator*{\Id}{{Id}}

\def \wto{\rightharpoonup}

\def\XXint#1#2#3{{\setbox0=\hbox{$#1{#2#3}{\int}$}
\vcenter{\hbox{$#2#3$}}\kern-.5\wd0}}

\usepackage{color}
\definecolor{verylightblue}{rgb}{0.95, 0.95, 0.95}  
\definecolor{lightblue}{rgb}{0.7, 0.7, 1}

\definecolor{eqyellow}{rgb}{0.9375,0.8984,0.5469}
\definecolor{subeqyellow}{rgb}{1,0.9373,0.8353}

\definecolor{mygreen}{rgb}{0.3, 0.6, 0.3} 
\definecolor{verylightgreen}{rgb}{0.95, 0.95, 0.95} 
\definecolor{verydarkgreen}{rgb}{0, 0.5, 0}
\definecolor{darkgreen}{rgb}{0.85, 0.85, 0.85}  
\definecolor{mydarkgreen}{rgb}{0, 0.5, 0} 

\definecolor{mybrown}{rgb}{0.85, 0.4, 0.3}
\definecolor{verylightbrown}{rgb}{0.98, 0.72, 0.58}
\definecolor{verydarkbrown}{rgb}{0.44, 0.26, 0.26}

\definecolor{orange}{rgb}{1, 0.5, 0}

\definecolor{BurntOrange}{rgb}{1,0.356,0}

\definecolor{mydarkred}{rgb}{1,0.086,0.255}

\definecolor{RoseVYDP}{rgb}{0.84,0.086,0.255}

\definecolor{dgreen}{rgb}{0, 0.8, 0.5}     

\definecolor{CanaryBRT}{rgb}{1,0.76,0.26}

\definecolor{cyan}{rgb}{0, 1, 1}
\definecolor{verylightgray}{rgb}{0.95, 0.95, 0.95}
\definecolor{verylightgray}{rgb}{0.95, 0.95, 0.95}
\definecolor{verylightred}{rgb}{1, 0.8, 0.78}
\definecolor{verylightyellow}{rgb}{0.99, 0.98, 0.5}

\newcommand{\ignore}[1]{{}}

\newcounter{margcount} 

\renewcommand{\S}{\mathcal S}
\newcommand{\T}{\mathcal T}

\newcommand{\md}{{\rm \mbox{d}}}
\renewcommand{\O}{\Omega}
\newcommand{\be}{\begin{eqnarray}} \newcommand{\ee}{\end{eqnarray}}

\title{A sharp interface evolutionary model for shape memory alloys} %
\author[1]{Hans Kn\"{u}pfer} %
\author[2,3]{Martin Kru\v{z}\'{\i}k} %
\affil[1]{\small University of Heidelberg, Institute of Applied Mathematics and IWR, Im Neuenheimer Feld 294, 69120 Heidelberg, Germany} %
\affil[2]{Institute of Information Theory and Automation, of the CAS, Pod vod\'{a}renskou v\v{e}\v{z}\'{\i}~4, CZ-182 08 Prague 8, Czech Republic,} %
\affil[3]{Faculty of Civil Engineering, Czech Technical University, Th\'akurova 7, 166 29 Prague, Czech Republic}

\begin{document}

\maketitle

\begin{abstract}
  We show the existence of an energetic solution to a quasistatic evolutionary model of shape memory alloys. Elastic
  behavior of each material phase/variant is described by polyconvex energy density. Additionally, to every phase
  boundary, there is an interface-polyconvex energy assigned, introduced by M.~\v{S}ilhav\'{y} in \cite{Silhavy-2010}.
  The model considers internal variables describing the evolving spatial arrangement of the material phases and a
  deformation mapping with its first-order gradients.  It allows for injectivity and orientation-preservation of
  deformations. Moreover, the resulting material microstructures have finite length scales.

	\bigskip
	
	
\end{abstract}

\section{Introduction}

In elasticity theory, it is assumed that experimentally observed patterns are minimizers or stable states of some
energy.  Shape memory alloys in particular have a preferred high-temperature lattice structure called austenite and a
preferred low-temperature lattice structure called martensite.  Such shape memory alloys, as e.g. Ni-Ti, Cu-Al-Ni or In-Th,
have various technological applications, for an overview see e.g. \cite{MohdLearyEtal-2014}.  The austenitic phase has
only one phase/variant but the martensitic phase exists in many symmetry related phases/variants; the mixing of these
different phases can lead to the formation of complex microstructure.  In the continuum theory, the total energy of the
system is described in terms of a bulk energy which describes elastic stresses and an interfacial energy, concentrated
on the interfaces between the different phases. We establish existence of quasistationary solutions for a model, where
it is assumed that the bulk part of the energy is polyconvex while the interfacial part of the energy satisfies a
corresponding condition of interfacial polyconvexity introduced by \v{S}ilhav\'{y} \cite{Silhavy-2010,Silhavy-2011}. The
model describes the evolving spatial arrangement of the material phases and the deformation of the sample.  It allows
for injectivity and orientation-preservation of deformations. Moreover, the resulting material microstructures have
finite length scales.

\medskip

To investigate the existence of a global minimizer of the energy for static variational problems from elasticity,
different notions of convexity have been considered.  For problems with a single material phase, a well justified notion
of convexity which is sufficient to ensure the existence of a minimizer is the notion of polyconvexity due to Ball
\cite{Ball-1977,Ball-1981}. It is also fairly easy to construct examples of polyconvex functions which makes it attractive for continuum mechanics of solids.  On the other hand, in shape memory alloys, many different phases might coexist.  If
interfacial energy is not taken into account, then global minimizers of the energy in general do not exist. A way out is
to use relaxation methods, searching for the so-called quasiconvex envelope of the specific stored energy
\cite{Dacorogna-Book,Mueller-Notes} or using Young measures
\cite{KinderlehrerPedregal-1991,KinderlehrerPedregal-1994, KruzikMielkeRoubicek-2005,MielkeRoubicek-2003}. Let us point out some partial results
which have been obtained in this direction: We refer to \cite{BenesovaKruzik-2015} for a weak* lower semicontinuity
results for sequences of bi-Lipschitz orientation-preserving maps in the plane and to \cite{BenesovaKampschulte-2014}
for an analogous result along sequences of quasiconformal maps. Then \cite{KoumatosRindlerWiedemann-2013} found
relaxation including orientation preservation for $p< d$, where $d$ is a spatial dimension. Finally,
\cite{ContiDolzmann-2015} derived a relaxation result for orientation preserving deformations with an extra assumption
on the resulting functional, namely that the quasiconvex envelope is polyconvex. There also exist various phenomenological models 
of shape memory alloys which are convenient for  numerical computations; see e.g.~\cite{BessoudKruzikStefanelli-2013}. 

\medskip

On the other hand, models have been considered where interfacial energy is taken into account. Such models have been
e.g.  used to estimate the scaling of the minimal energy and to derive typical length scales of patterns. The minimal
scaling of the energy of an austenite-martensite interface has been studied by Kohn \& Müller and Conti in
\cite{KohnMueller-1992-2,KohnMueller-1994,Conti-2006} for a 2-d model problem, the three-dimensional case and more
realistic models have been investigated e.g. in
\cite{Zwicknagl-2014,ChanConti-2015,CapellaOtto-2009,CapellaOtto-2010,KnuepferKohn-2011,KnuepferKohnOtto-2012}, for
similar analysis on related models see e.g. \cite{ChoksiKohn-1998,ChoksiKohnOtto-1999,KnuepferMuratov-2011}. In these
models, either a $BV$-penalization of the interfacial has been used or a penalization of some $L^p$--norm for the
Hessian of the deformation function. In general, the specific form of the energy is, however, not clear from physical
considerations. In the literature, necessary and sufficient conditions for the specific form of the interfacial energy
have been investigated recently which allow for the existence of minimizers \cite{Fonseca-1989,Parry-1987}. Recently,
\v{S}ilhav\'{y} has introduced a notion of interface polyconvexity and has proved that this notion is sufficient to
ensure existence of minimizers for the corresponding static problem \cite{Silhavy-2010,Silhavy-2011}. In this note, we
extend this static model to a rate-independent evolutionary model and prove existence of an energetic solution.

\medskip

In shape memory allows, the stored energy density $W : \R^{3 \times 3} \to \R$ is minimized on wells $\SO(3)F_i$,
$i=0,\ldots, M$, defined by $M$ positive definite and symmetric matrices $F_0,\ldots, F_M$, each corresponding austenite
and $M$ variants of martensite, respectively. By the choice of reference configuration, we may furthermore assume $F_0:= \Id$ (the identity),
i.e. the stress-free strain of austenite is described just by the special orthogonal group $\SO(3)$. In nonlinear
elasticity, the energy density $W$ is usually formulated as a function of the right Cauchy-Green strain tensor $F^\top
F$. Note that this tensor maps the whole group $O(3)$ of orthogonal matrices with determinant $\pm 1$ onto the same
point. Thus, for example, $F\mapsto |F^\top F-\Id|$  is minimized on two
energy wells, i.e., on SO$(3)$ and also on O$(3)\setminus$SO$(3)$.  However, the latter set is not acceptable in
elasticity since corresponding deformations do not preserve the orientation. Additionally, notice that, for example,
considering arbitrary $Q\in {\rm O}(3)\setminus {\rm SO}(3)$ and an arbitrary $R\in \SO(3)$ such that $Q$ and $R$ are
rotations around the same axis of the Cartesian system then rank$(Q-R)=1$, i.e. $Q$ and $R$ are rank-one connected and determinant changes its sign on the line segment $[Q;R]$. Convex combinations of rank-one connected
matrices play a key role in relaxation approaches of the variational calculus \cite{BallJames-1987,BallJames-1992,Dacorogna-Book,KruzikLuskin-2003}.    This shows that is it important but also not
straightforward to ensure that  solutions of static or evolutionary problems are physically sound, in the sense that they
preserve orientation.  This is, in particular, unclear on models based on quasiconvexification as described above, since
usually there is no closed formula of the envelope at disposal and since physically justified conditions on deformations
as orientation-preservation and injectivity are not included in these models. On the other hand, our solutions are
constructed in a way such that the obtained time dependent deformations are orientation preserving and injective and no
additional regularization of variables is needed if passing from a static to an evolutionary model. Injectivity and orientation preservation is  not very often  considered in the theory of rate-independent processes. We refer to \cite{MielkeRoubicek-preprint} for treatment  of a model in nonlinear elastoplasticity.    

\medskip

\textit{Structure of the paper:} After introducing used notations, in Section 2, we first describe our model, the stored
elastic energy, loading, and dissipation. In Section 3, we state and proof our main result, the existence of an
energetic solution. As it is nowadays a standard procedure; cf.~e.g., \cite{FrancfortMielke-2006} we only sketch the
main steps and pay more attention to injectivity of deformations which is not frequently treated in the frameworks of
rate-independent evolutions. We refer, however,  to \cite{MielkeRoubicek-preprint} for numerical approaches to finite
elastoplasticity including injectivity.

\medskip

\textit{Notation:} The spaces $W^{1,p}$, $1 \leq p < \infty$, denote the standard Sobolev space of $L^p$-functions with
weak derivative in $L^p$. Further, $BV$  stands for the space of
integrable maps with bounded variations, see e.g. \cite{AmbrosioFuscoPallara-Book,EvansGariepy-Book} for references. For
a (measurable) set $E \subset \R^3$, we denote its three-dimensional Lebesgue measure by $\LL^3(E)$ and its
two-dimensional Hausdorff measure by $\HH^2(E)$. The space of vector valued Radon measures on $\Ome$ with values in
  $Y$ is denoted by $\MM(\Ome,Y)$.

\medskip

Let $\tilde\Ome\subset\Ome\subset\R^3$ be Lebesgue measurable sets and let
$B(x,r):=\{a\in\R^3:\, |x-a|< r\}$ . For $x\in\Ome$ we denote the the {\it density} of $\tilde\Ome$ at $x$ by
$\theta(\tilde\Ome,x):=\lim_{r\to 0}\mathcal{L}^3(\tilde\Ome\cap B(x,r))/\mathcal{L}^3(B(x,r))$ whenever this limit
exists. A point $x \in \Ome$ is called {\it point of density} of $\t \Ome$ if $\theta(\tilde\Ome,x)=1$.  If
$\theta(\tilde\Ome,x)=0$ for some $x\in\Ome$, then $x$ is called {\it point of rarefaction} of $\tilde\Ome$. The
measure-theoretic boundary $\partial^*\tilde\Ome$ of $\tilde\Ome$ is the set of all points $x\in\Ome$ such that either
$\theta(\tilde\Ome,x)$ does not exist or $\theta(\tilde\Ome,x) \nin \{ 0, 1 \}$. We call $\tilde\Ome$ a set of finite
perimeter if $\mathcal{H}^{2}(\partial^*\tilde\Ome)<+\infty$. Let $n\in\R^3$ be a unit vector and let $H(x,n):=\{\tilde
x\in\Ome:\, (\tilde x-x)\cdot n<0\}$. We say that $n$ is the (outer) measure-theoretic normal to $\tilde\Ome $ at $x$ if
$\theta(\tilde\Ome\cap H(x,-n),x)=0$ and $\theta((\Ome\setminus\tilde\Ome)\cap H(x,n),x)=0$.  The measure-theoretic
normal exists for $\HH^2$ almost every point in $\partial^*\tilde\Ome$, see e.g. \cite{EvansGariepy-Book, Silhavy-Book}.

\medskip

For two matrices $A = (a_{ij}), B = (b_{ij}) \in \R^{3\times 3}$, we define $A:B = a_{ij}b_{ij}$ with Einstein's sum
convention.  By $A \times n$ we denote the tensor defined by $(A \times n) b = A (n \times b)$, i.e. $(A \times n)_{kj}
= \eps_{\ell ij} a_{k\ell} n_i$, where $\eps_{\ell ij}$ is the Levi-Civita symbol. One can easily check that the cofactor matrix of $A
\in \R^{3\times 3}$ in terms of the Levi-Civita can be expressed as $\cof A$
   $= \frac 12 (\eps_{ik\ell} \eps_{jpq} a_{kp} a_{\ell q})_{ij}$.
   In particular, we get $\p_{a_{k \ell}} (\cof A)_{ij}$ $=$ $\frac 12 \p_{a_{k \ell}} (\eps_{ik\ell} \eps_{jpq} a_{kp}
   a_{\ell q})_{ij}$ $=$ $\eps_{ikq} \eps_{j\ell p} a_{qp}$. We refer e.g. to \cite{GurtinStruthers-1990} for a
   definition of the surface gradients $\nabla _S$. If $n\in\R^3$ is an outer unit normal to the surface $S$ , then
   $\nabla_S:=\nabla(\Id-n\otimes n)$, where  we recall that $\Id$ denotes the unit matrix in $\R^{3\times 3}$.

\medskip

\section{Model description}

\subsection{Elastic energy} \label{ss-static} %

{\it Admissible States: } We assume that the specimen in its reference configuration is represented by a bounded
Lipschitz domain $\Ome\subset\R^3$. We consider a shape memory alloy which allows for $M$ different variants of
martensite. The region occupied by the $i$-th variant of martensite is described by the set $\Ome_i \subset \Ome$ for $1
\leq i \leq M$, while the region occupied by austenite is given by $\Ome_0 \subset \Ome$.  We assume that the sets
$\Ome_i$ are open and have finite perimeter.  Furthermore, the sets $\Ome_i$ are pairwise disjoint for $0\leq i\leq M$
and $N := \Ome \BS \bigcup_i \Ome_i$ is a set of zero Lebesgue measure. The case $\Ome_i = \emptyset$ for some $0 \leq
i \leq M$ is not excluded. The partition of $\Ome$ into $\{\Ome_i\}_{i=0}^M$ can be then identified with a mapping
$z:\Ome \to \R^{M+1}$ such that $z_i(x)=1$ if $x\in\Ome_i$ and $z_i(x)=0$ else.  We call $z$ the {\it partition map}
corresponding to $\{\Ome_i\}_{i=0}^M$. Clearly, with the sets $\Ome_i$ chosen as before, we have $\sum_{i=0}^Mz_i(x)=1$
for almost every $x\in\Ome$ and the function $z$ is of bounded variation. We hence consider $z \in \ZZ$, where
\begin{align} \NT %
  \ZZ := \Big \{ z \in {\rm BV}(\Ome,\{ 0, 1 \}^{M+1}) : \ &z_i z_j = 0 \text{ for $i \neq j$, } \sum_{i=0}^M z_i = 1
  \text{ a.e.~in $\Ome$} \Big \}.
\end{align}
In order to describe the state of the elastic material, we also need to introduce the deformation function $y \in
W^{1,p}(\Ome,\R^3)$, $p > 3$, which describes the deformation of the elastic body with respect to the reference
configuration $\Ome$. We hence consider deformations $y \in \YY$, where
\begin{align} \NT %
  \YY = \Big\{ y \in W^{1,p}(\Ome,\R^3) \ : \ \det \nabla y >0 \text{ a.e.}\ , \int_\Ome \det\nabla y(x)\,\md x\leq \mathcal{L}^3(y(\O)) \Big\}, \NT
\end{align}
where we will always use the assumption $p > 3$. The integral inequality together with the orientation-preservation is
the so-called Ciarlet-Ne\v{c}as condition which ensures invertibility of $y$ almost everywhere in $\O$
\cite{Ciarlet-Book,CiarletNecas-1987}. In the following, we will assign to each state of the material $(y,z) \in \YY
\times \ZZ$ an elastic energy $\EE$.  In our model, the energy consists of a bulk part $E_{\rm b}$, penalizing deformation
within the single phases, an interfacial energy $E_{\rm int}$, measuring deformation of the interfaces between the
phases and a contribution $L(t,\cdot)$ which measures work of external loads, i.e.
\begin{align}
  \EE(t,y,z):=E_{\rm b}(y,z)+E_{\rm int}(y,z)-L(t,y).
\end{align}
Here $t$ denotes time to indicate that we will deal with time-dependent problems.
We will specify these three parts of the energy in the following.

\medskip

{\it Bulk energy:} The total bulk energy of the specimen has the form
\begin{align} \label{stat-bulk} %
  E_{\rm b}(y,z):=\int_\Ome W(z(x),\nabla y(x))\ dx,
\end{align}
where we assume that the specific energy $W : \R^{M+1} \times \R^{3 \times 3} \to \R\cup\{+\infty\}$ of the specimen can be
written as
\begin{align}
  W(z,F) := \sum_{i=0}^M z_i \hat W_i(F) =:z \cdot \hat W(F),
\end{align}
where $\hat W_i$, $0 \leq i \leq M$, is the specific energy related to the $i$-th phase of the material and $\hat
W:=(\hat W_0,\ldots, ,\hat W_M)$. We will work in the framework of hyperelasticity, where the first Piola-Kirchhoff
stress tensors of austenite and martensite have polyconvex potentials denoted by $\hat W_0$ (austenite) and $\hat W_i$,
$i=1,\ldots, M$ for each variant of martensite, see e.g. \cite{Silhavy-2010} and the
references therein.  For $0\leq i\leq M$, we therefore assume
\begin{align} \label{W-ass1} %
  \hat W_i(F):=
  \begin{cases}
    h_i(F,\cof F,\det F) & \mbox{ if } \det F>0, \\
    +\infty \mbox{ otherwise}
  \end{cases}
\end{align}
for some convex functions $h_i:\R^{19} \to \R$. We use the following additional standard assumptions on the specific bulk
energies $\hat W_i$. For $0\leq i\leq M$ and $F\in\R^{3\times 3}$, we assume that for some $C>0$ and $p>3$
\begin{align}
  &\hat W_i(F)\ge C(-1+|F|^p) &&\forall F \in \R^{3 \times 3}\ , \label{W-ass2} \\
  &\hat W_i(RF)=\hat W_i(F) && \forall R\in{\rm SO}(3), F \in \R^{3\times 3} \ , \label{W-ass3}\\
  &\lim_{\det F\to 0_+} \hat W_i(F)=+\infty\ . \label{W-ass4}
\end{align}

\medskip

{\it Interfacial energy:} We consider the interfacial energy in the form introduced by \v{S}ilhav\'{y} in
\cite{Silhavy-2010,Silhavy-2011}: We hence assume that the specific interfacial energy $f_{ij}$ between the two
different phases $i, j\in\{0,\ldots, M\}$ can be written in the form
\begin{align} \label{intcon-1} %
  \frac12 f_{ij}(F, n)=g_i(F,n)+g_j(F,n),
\end{align}
where $F\in\R^{3\times 3}$ and $n\in\R^3$ is a unit vector such that $F n=0$,
We assume
\begin{align} \label{intcon-2} %
  g_i(F,n):= \Psi_i(n,F\times n, \cof F\, n),
\end{align}
where the functions $\Psi_i: \R^{15} \to \R$ are nonnegative convex and positively one-homogeneous for $i=0,
\ldots, M$. Here, $F\times n:\R^3\to \R^3$ is for any $F\in\R^{3\times 3}$ and any $n, a \in\R^3$ defined as $(F\times
n)a:=F(n\times a)$. As in \cite{Silhavy-2010}, we assume for $0 \leq i \leq M$, $ \forall F\in\R^{3\times 3}$, $\forall
n\in S^2$
\begin{align}
  g_i(RF,n) &=g_i(F,n) \ \qquad  \forall R\in \SO(3),  \label{g-ass2} \\
  g_i(F,n) &=g_i(F,-n), \label{g-ass3}
\end{align}
As in \cite{Silhavy-2010}, we assume that there is some $c>0$ such that
\begin{align}
  \Psi_i(A)\ge c|A| \ \label{g-ass1}\ .
\end{align}
for all $0\leq i\leq M$ and all $A\in \R^{15}$. We introduce a subspace $\QQ \subset \YY \times \ZZ$ of functions with
``finite interfacial energy'', using a slightly modified version of \cite[Def. 3.1]{Silhavy-2010}. It is given as
follows:
\begin{definition}[Interfacial energy] \label{def-intok} %
  For any pair $(y,z) \in \YY \times \ZZ$ let $S_i=\partial^*\Ome_i \cap\Ome$ where $\Ome_i := \supp z_i$ and
  $\partial^*\Ome_i$ is the measure-theoretic boundary of $\Ome_i$ with outer (measure-theoretic) normal $n_i$. We
  denote by $\QQ \subset \YY \times \ZZ$ the set of all pairs $(y,z) \in \YY \times \ZZ$ such that for every $0 \leq i
  \leq M$ there exists a measure $J_i:=(a_i,H_i,c_i)\in \MM(\O;\R^{15})$ with
  \begin{align}\label{measures1}
    a_i := n_i \HH^2_{|S_i}, && %
    H_i:= \nabla_{S_i} y\times n_i \HH^2_{|S_i} && \text{and } && %
    c_i:= (\cof\nabla_{S_i} y) n_{|S_i}\ .
  \end{align}
   The interfacial energy is then defined as
  \begin{align} \label{stat-interface} E_{\rm int}(y,z):=
    \begin{TC}
      \displaystyle \sum_{i=0}^M \int_\Ome \Psi_i\left(\frac{\md J_i}{\md |J_i|}\right)\ \md|J_i|\ %
      &\text{for $(y,z) \in \QQ$,} \\
      \infty &\text{else.}
    \end{TC}
  \end{align}
  Here $|J_i|$ denotes the total variation of the measure $J_i$.
\end{definition}
We recall that the function $f_{ij}$ is called interface quasiconvex if
  \begin{align} \label{int-qc} %
    \int_{\S} f(\nabla_\S y,n) d\HH^2 \geq \HH^2(\mathcal T) f(G,m)
  \end{align}
  for every surface deformation gradient $G$, every unit vector $m$ with $Gm = 0$, every planar two-dimensional region
  $\mathcal T$ with normal $m$, every (curved) surface $\S$ with normal $n$ and every smooth map $y : \S \to \R^3$ with
  ${\rm bd} \ = {\rm bd} \T$ (where ${\rm bd} \S = {\rm bd} \T$ denote the relative boundaries of the two
  two-dimensional surfaces) and such that $y = Gx$ for $x \in {\rm bd} \T$, see \cite{Silhavy-2010,Silhavy-2011}. A
  surface energy is called Null-Lagrangian if \eqref{int-qc} is satisfied with equality. Furthermore, it has been shown
  in \cite{Silhavy-2010b} that $f$ is an interface Null-Lagrangian if and only if $f$ is a linear function of $n$, $F
  \times n$ and $\cof Fn$. This motivates the definition of interface polyconvexity \eqref{intcon-1}--\eqref{intcon-2},
  in the analogy to the definition of the standard notion of polyconvexity. The set of configurations $\QQ$ in
  Definition \ref{def-intok} is the natural space where an energy of type \eqref{intcon-1}--\eqref{intcon-2} can be
  defined. Let us remark that the measures $H_i$ and $c_i$ can be expressed as
\begin{align} \label{form-1} %
  \int_\Ome v \ \md H_i = \int_{\Ome_i} \nabla y\ (\nabla \times v) \ dx, && \int_\Ome v\cdot \md c_i = \int_{\Ome_i}
  (\cof\nabla y) : \nabla v\ dx
\end{align} %
for all $v\in C^\infty_0(\O;\R^3)$. Indeed, for $k \in \{ 1, 2,3 \}$ and $0\le i\le M$, we have
\begin{align*}
  \int_{\Ome_i} [\nabla y (\nabla \times v)]_k %
  &= \int_{\Ome_i} [\p_j y_k \eps_{j\ell m} \p_\ell v_m] dx %
  =   \int_{\p \Ome_i} [n_\ell \p_j y_k \eps_{j\ell m}  v_m ]  dx \\
  &= \int_{\p \Ome_i} [\nabla v]_{kj} [n \times v]_j dx = \int_{\p \Ome_i} [(\nabla y \times n) v]_k dx,
\end{align*}
since $\nabla \times \nabla y = 0$.  With the notation $(\cof \nabla y)_{ij} = b_{ij}$, we also have
\begin{align*}
  \int_{\Ome_i} (\cof \nabla y) : (\nabla v) dx %
  &= \int_{\Ome_i} [b_{kj} \p_j v_k] dx %
  =   - \int_{\Ome_i} [\p_j b_{kj}   v_k]  dx + \int_{\p \Ome_i} [n_j b_{kj}  v_k ]  dx \\
  &= \int_{\p \Ome_i} [(\cof \nabla y) n]_{k} v_k dx = \int_{\p \Ome_i} (\cof \nabla y) n \cdot v dx,
\end{align*}
where we used the Piola identity $\nabla \cdot (\cof \nabla y) = 0$.

\medskip

We also note that by the assumption \eqref{g-ass1}, we have the bound
\begin{align}
  \|D z\|_{\mathcal{M}(\O;\R^{(M+1)\times 3})} \ \leq \ C E_{\rm int}(y,z).
\end{align}
for some constant $C < \infty$.  Consequently, the norm $\|z\|_{{\rm BV}(\O;\R^{M+1})}$ is controlled in terms of the
interfacial energy in our setting. On the other hand, the norm $\|D z\|_{\mathcal{M}(\O;\R^{(M+1)\times 3})}$
  satisfies the conditions in Definition \ref{def-intok}. Indeed, this follows from the choice
  $g_i(F,n)=\alpha|F|=\alpha|F\times n|$ for $\alpha>0$. Another example of an interfacial energy which is included in
  the Definition \eqref{def-intok} is given by the choice $g_i(F,n)=\alpha |\cof F n|$, see \cite{Silhavy-2011} for more
  details. Notice that the first example penalizes surface gradients which are nonconstant along interfaces while the latter one 
	increases with the area of the interface. 

\medskip

{\it Body and surface loads:} We assume that the body is exposed to possible body and surface loads, and that it is
elastically supported on a part $\Gamma_0$ of its boundary. The part of the energy related to this loading is given by a
functional $L\in C^1([0,T]; W^{1,p}(\O;\R^3))$ in the form
\begin{align} \label{staticloading} %
  L(t,y) := &\int_\Ome b(t) \cdot y \ dx+\int_{\Gamma_1} s(t) \cdot y \ \md \HH^2(x) %
  +\frac{K}{2}\int_{\Gamma_0}|y-y_D(t)|^2\,\md \HH^2(x).
\end{align}
Here, $b(t,\cdot):\Ome \to \R^3$ represents the volume density of some given external body forces and
$s(t,\cdot):\Gamma_1\subset\partial\Ome \to \R^3$ describes the density of surface forces applied on a part $\Gamma_1$
of the boundary. The last term in \eqref{staticloading} with $y_D(t,\cdot)\in W^{1,p}(\O;\R^3)$ represents energy of a
spring with a spring stiffness constant $K>0$. Thus our specimen is elastically supported on $\Gamma_0$ in such a way,
that for $K\to \infty$ $y$ is forced to be close to $y_D$ on $\Gamma_0$ in the sense of the $L^2(\Gamma_0;\R^3)$ norm.
A term of this type already appeared in \cite{KruzikOtto-2004} and its static version also in
\cite{MielkeRoubicek-2003}.  Namely, prescribing a boundary condition from $W^{1-1/p,p}(\partial\Ome;\R^3)$
\cite{Marschall-1987}, it is generally not known whether it can be extended to the whole $\Omega$ in such a way that the
extension lives in $\YY$. It is, to our best knowledge, an unsolved problem in three dimensions and therefore it is
generically assumed in nonlinear elasticity that such an extension exists; cf.~\cite{Ciarlet-Book}, for instance.  The
last term in \eqref{staticloading} overcomes this drawback. Namely, if $y_D$ cannot be extended from the boundary as an
orientation-preserving map the term in question will never be zero regardless values of $K>0$.

%

\medskip

\subsection{Dissipation}
Evolution is typically connected with dissipation of energy. Experimental evidence shows that it is a reasonable
approximation in a wide range of rates of external loads to anticipate  a rate-independent dissipation mechanism. In order
to set up such a process, we need to define a suitable dissipation function. Since we consider  rate-independent
processes, this dissipation will be positively one-homogeneous.  We associate the dissipation to the magnitude of the
time derivative of $z$, i.e., to $|\dot z|_{M+1}$, where $|\cdot|_{M+1}$ is a norm on $\R^{M+1}$. Therefore, the specific
dissipated energy associated to a change of the variant distribution from $z^1$ to $z^2$ is postulated as in \cite{DeSimoneKruzik-2013}
\begin{align}\label{dissipation}
  D(z^1,z^2):=|z^1-z^2|_{M+1}.
\end{align}
Then the total dissipation reads
\begin{align*}
  \mathcal{D}(z^1,z^2):=\int_\Ome D(z^1(x),z^2(x))\ dx \ .
\end{align*}
The $\DD$-dissipation of a curve $z : [0,T] \to BV(\Ome,\{ 0,1 \})$ with $[s,t] \subset [0,T]$ is correspondingly given
by  (see e.g.~\cite{FrancfortMielke-2006})
\begin{align*}
  {\rm Diss}_\DD(z,[s,t]) := \sup \Big \{ \sum_{j=1}^N \DD(z(t_{i-1}),z(t_i)) : N \in \N, s = t_0 \leq \ldots \leq t_N =
  t \Big \}.
\end{align*}

\subsection{Energetic solution}

\medskip

Suppose, that we look for the time evolution of $t\mapsto y(t)\in \YY$ and $t\mapsto z(t)\in \ZZ$ during a process time
interval $[0,T]$ where $T>0$ is the time horizon. We use the following notion of solution from
\cite{FrancfortMielke-2006}, see also \cite{MielkeTheil-2004,MielkeTheilLevitas-2002}: For every admissible
configuration, we ask the following conditions to be satisfied for all $t \in [0,T]$.
\begin{definition}[Energetic solution]
  We say that $(y,z) \in \YY\times \ZZ$ is an energetic solution to $(\EE,\DD)$ on the time interval $[0,T]$ if
  $t \mapsto \p_t E(y(t),z(t)) \in L^1((0,T))$ and if for all $t \in [0,T]$, the stability condition
  \begin{align} \label{stability} %
    &\EE(t, y(t),z(t))\leq \EE(t,\tilde y,\tilde z)+\mathcal{D}(z(t),\tilde z) \qqquad %
    \text{$\forall (\tilde y,\tilde z)\in\QQ$}.
  \end{align}
  and the condition of energy balance
  \begin{align} \label{energy-balance} %
    \begin{aligned}
      &\EE(t, y(t),z(t))+{\rm Diss}_\DD(z;[0,t]) %
      =\EE_0 +\textstyle \int_0^t \frac{\partial \EE}{\partial t}(s, y(s),z(s))\,\md s \hspace{-2ex}
    \end{aligned}
  \end{align}
  where $\EE_0 = \EE(0,y(0),z(0))$, are satisfied.
\end{definition}
An important role in the theory of rate-independent solutions is played by the so-called stable states defined for each
$t\in[0;T]$. We set
\begin{align*}
  \S(t):=\{(y,z)\in\YY\times \ZZ:\, \EE(t, y,z)\leq \EE(t,\tilde y,\tilde z)+\mathcal{D}(z,\tilde z)\,\forall (\tilde
  y,\tilde z)\in\QQ\}.
\end{align*}
Note that by \eqref{stability}, any energetic solution $(y,z)$ is stable for any fixed time.

\section{Existence of the energetic solution}

A standard way how to prove the existence of an energetic solution is to construct time-discrete minimization problems
and then to pass to the limit. Before we give the existence proof we need some auxiliary results. For given $N\in\N$ and
for $0\leq k\leq N$, we define the time increments $t_k:=kT/N$. Furthermore, we use the abbreviation $q:=(y,z)\in\QQ$.
Assume that at $t=0$ there is given an initial distribution of phases $z^0\in\ZZ$ and $y^0\in\YY$ such that $q^0=(y^0,
z^0)\in \S(0)$.  For $k=1,\ldots, N$, we define a sequence of minimization problems
\begin{align}\label{incremental}
  \text{minimize } \EE(t_k, y,z)+\DD(z,z^{k-1})\ ,\ (y,z)\in \QQ\ .
\end{align}
We denote a minimizer of \eqref{incremental} for a given $k$ as $(y^k,z^k)\in\QQ$.  The following proposition shows that
a minimizer always exists if the elastic energy is not identically infinite on $\QQ$.

\begin{lemma} \label{lem-dav} %
  Assume that $p>3$, \eqref{W-ass1}-\eqref{W-ass4}, \eqref{intcon-2}, \eqref{g-ass3}-\eqref{g-ass1} hold and let $L\in
  C^1([0,T]; W^{1,p}(\Omega;\R^3))$.  Let $q^N_0 := (y_0,z_0)\in\QQ$
  satisfy $\EE(0,y,z)<+\infty$.  Then there exists a solution $q^N_k := (y_k,z_k)$ to \eqref{incremental} for each
  $1\leq k\leq N$.  Moreover, $q^N_k\in \S(t_k)$ for all $1 \leq k \leq N$.  
\end{lemma}
\begin{proof}
  The proof follows the same lines as the proof of \cite[Thm.~3.3]{Silhavy-2010}. We apply the direct method of the
  calculus of variations. We denote the elements of the minimizing sequence by a lower index in brackets in order to
  distinguish it from the components of $z=(z_0,\ldots, z_M)$.  Fix $k$, so that $z^{k-1}\in\ZZ$ is given. Let
  $\{(y_{(j)}, z_{(j)})\}_{j\in\N}\subset\QQ$ be a minimizing sequence for $\EE(t_k, \cdot,\cdot)+\DD(\cdot,z^{k-1})$.
  Using the growth conditions \eqref{W-ass2}, \eqref{g-ass2}, and in view of the form of $L$, it follows that there is
  $C>0$ such that $\|y_{(j)}\|_{W^{1,p}(\Ome;\R^3)}+\|z_{(j)}\|_{{\rm BV}(\Ome;\R^{M+1})}\leq C$ for all
  $j\in\N$. Furthermore, 
  \begin{align*}
    \sup_{j}(\|\cof \nabla y_{(j)}\|_{L^{p/2}(\Ome;\R^{3\times 3})} + \|\det\nabla y_{(j)}\|_{L^{p/3}(\Ome)})<+\infty,
  \end{align*}
  where $p/3>1$ by our assumption $p > 3$.  Consequently, after taking a subsequence, we may assume that
  $y_{(j)}\wto y$ in $W^{1,p}(\Ome;\R^3)$, $\det\nabla y_{(j)}\wto \det\nabla y$ in $L^{p/3}(\Ome)$, $\cof\nabla
  y_{(j)}\wto \cof\nabla y$ in $L^{p/2}(\Ome;\R^{3\times 3})$, and $z_{(j)}\stackrel{*}\wto z$ in ${\rm
    BV}(\O;\R^{M+1})$.  In particular, we have $z_{(j)}\to z$ in $L^1(\O;\{0,1\}^{M+1})$ and $z\in \ZZ$. Moreover, in
  view of \eqref{form-1}, $(J_i)_{(j)}$ converges weakly* in measures to $J_i$ as $j\to \infty$ for all $0\le i\leq M$ .
  Standard results for polyconvex materials \cite{Ball-1977, Ciarlet-Book,Silhavy-2010} show $\liminf_{j\to
    \infty}E_b(y_{(j)},z_{(j)})\ge E_b(y,z)$. Similarly, $\liminf_{j\to \infty}L(t_k, y_{(j)})\ge L(t_k, y)$ and
  $\lim_{j\to \infty}\DD(z_{(j)},z^{k-1}) =\DD(z,z^{k-1})$ due to the strong convergence of $z_{(j)}\to z$ in
  $L^1(\Ome;\R^{M+1})$.  Finally,
  \begin{align*}
    \liminf_{j\to \infty} E_{\rm int}(y_{(j)},z_{(j)})\ge E_{\rm int}(y,z)
  \end{align*}
  due to \cite[Thm.~2.38]{AmbrosioFuscoPallara-Book}.  Thus, $(y,z)\in\YY\times \ZZ$. Using weak sequential continuity of $y\mapsto\cof\nabla y$ and $y\mapsto\nabla y$ we see that 
	the limiting measures $J_i$ have the form of \eqref{measures1}. This together with  a limit passage in the Ciarlet-Ne\v{c}as condition (see 
  \cite[Thm.~5]{CiarletNecas-1987}) shows that  $(y,z)\in\QQ$. Namely, $y$ is injective almost everywhere in $\Ome$ and
  $\det\nabla y>0$ almost everywhere in $\Ome$. From \eqref{incremental}, one furthermore easily sees
    that $q^N_k\in\S(t_k)$ for all $1 \leq k \leq N$.
\end{proof}
Denoting by $B([0,T];\YY)$ the set of bounded maps $t\mapsto y(t)\in\YY$ for all $t\in[0,T]$, we have the following
result showing the existence of an energetic solution.
\begin{theorem}
  Let $T>0$, $p>3$, $y_{\rm D}\in C^1([0,T]; W^{1,p}(\O;\R^3))$, \eqref{W-ass1}-\eqref{W-ass4}, \eqref{intcon-2},
  \eqref{g-ass3}-\eqref{g-ass1}. Let $(y(0),z(0))\in S(0)$ and Let there be $(y,z)\in\QQ$ such that
  $\EE(0,y,z)<+\infty$.  Then there is and energetic solution to the problem $(\EE,\DD)$ such that $y\in
  B([0,T];\YY)$, $z\in {\rm BV}([0,T]; L^1(\O;\R^{M+1})\cap L^\infty(0,T;\ZZ)$.
\end{theorem}
\begin{proof}
  Let $q^N_k:=(y^k,z^k)$ be the solution of \eqref{incremental} which exists by Lemma \ref{lem-dav} and let
  $q^N:[0,T]\to \QQ$ be given by
  \begin{eqnarray}
    q^N(t):= %
    \begin{cases}
      q^N_k &\mbox{ if $t\in [t_{k},t_{k+1})$ if $k=0,\ldots, N-1$}\ ,\\
      q^N_N &\mbox{ if $t=T$.}
    \end{cases}
  \end{eqnarray} 
  Following \cite{FrancfortMielke-2006}, we get for some $C>0$ and for all $N\in\N$ the estimates
  \begin{subequations}
    \begin{gather}
      \|z^N\|_{BV(0,T; L^1(\O;\R^{M+1}))}\leq C, \qquad
      \|z^N\|_{L^\infty(0,T; BV(\O;\R^{M+1}))}\leq C, \\
      \|y^N\|_{L^\infty(0,T;W^{1,p}(\O;\R^3))}\leq C, 
    \end{gather}
  \end{subequations}
  as well as the following two-sided energy inequality
  \begin{align}
    \int_{t_{k-1}}^{t_k}\partial_t\mathcal{E}(\theta,q_{k}^N)\,{\rm d}\theta %
    &\leq \mathcal{E}(t_k,q^N_k)+\mathcal{D}(z^k,z^{k-1})-\mathcal{E}(t_{k-1},q^N_{k-1}) \NT \\
    &\leq \int_{t_{k-1}}^{t_k}\partial_t\mathcal{E}(\theta,q_{k-1}^N)\,{\rm d}\theta\ . \label{2-sided}
  \end{align}
  The second inequality in \eqref{2-sided} follows since $q_{k}^N$ is a minimizer of \eqref{incremental} and by
  comparison of its energy with $q:=q_{k-1}^N$. The lower estimate is implied by the stability of
  $q^N_{k-1}\in\mathbb{S}(t_{k-1})$ when compared with $\tilde q:=q^N_{k}$.  Having this inequality, the a-priori
  estimates and a generalized Helly's selection principle \cite[Cor.~2.8]{MielkeTheilLevitas-2002} we get that there is
  indeed an energetic solution obtained as a limit for $N\to \infty$. In particular, the fact that $\det\nabla y>0$
  a.e. in $\Omega$ follows from the fact that if $t_j\to t$, $(y_{(j)}, z_{(j)})\in\mathbb{S}(t_j)$ and $(y_{(j)},
  z_{(j)})\wto (y,z)$ in $W^{1,p}(\Ome;\R^3) \times BV(\O;\R^{M+1})$, then $(y,z)\in\mathbb{S}(t)$.  Indeed, in
  particular we have $z_{(j)}\to z$ in $L^1(\O;\R^{M+1})$ and hence for all $(\tilde y,\tilde z)\in\QQ$, we get
  \begin{align*}
    \EE(t,y,z) %
    &\leq \liminf_{j\to \infty}\EE(t_j,y_{(j)}, z_{(j)}) %
    \leq \liminf_{j\to \infty}\EE(t_j,\tilde y, \tilde z) +\liminf_{j\to \infty}\DD(z_{(j)},\tilde z)\\
    &=\EE(t,\tilde y, \tilde z)+\DD(z,\tilde z)\ .
  \end{align*}
  In particular, as $\EE(t_j,\tilde y, \tilde z)$ is finite for some $(\tilde y,\tilde z)\in\QQ$ we get
  $\EE(t,y,z)<+\infty$ and thus $\det\nabla y>0$ a.e. in $\Omega$ in view of \eqref{W-ass1}.
\end{proof}

\begin{remark}
Adding a term of the form $F\mapsto |\cof F|^p/(\det F)^{p-1}$, which is polyconvex, to the bulk stored energy density we can even show injectivity of deformations  everywhere in $\Omega$ for all time instants. See  e.g.~\cite[Rem.~1.2]{BenesovaKruzik-2015}  and also \cite{Ball-1981} where such term  already appeared. 
\end{remark}

{\bf Acknowledgment.} This work was initiated during a visit of MK in the Institute of Applied Mathematics at the
University of Heidelberg and further conducted during a visit of HK in the Institute of Information Theory and
Automation at Prague. The hospitality and support of both institutions is gratefully acknowledged. MK was further
supported by GA\v{C}R through the project 14-15264S.



\end{document}